\documentclass[reqno,11pt]{amsart}
\usepackage{amssymb}
\usepackage{verbatim}
\usepackage{color}
\usepackage[pdftex]{graphicx}
\usepackage[pdftex,hyperindex]{hyperref}
\usepackage[twoside=false,total={142mm,220mm}]{geometry}
\numberwithin{equation}{section}       
\numberwithin{figure}{section}       
\theoremstyle{plain}
\newtheorem{Thm}{Theorem}[section]
\newtheorem{Prop}[Thm]{Proposition}
\newtheorem{Lemma}[Thm]{Lemma}

\newtheorem{Prop-def}[Thm]{Proposition-Definition}

\newtheorem*{MainThm}{Main Theorem}

\theoremstyle{definition}

\newtheorem{Conj}[Thm]{Conjecture}
\newtheorem{Example}[Thm]{Example}

\newtheorem{Def}[Thm]{Definition}

\newcommand{\A}{{\mathbf{A}}}
\newcommand{\C}{{\mathbf{C}}}

\newcommand{\N}{{\mathbf{N}}}
\renewcommand{\P}{{\mathbf{P}}}
\newcommand{\Q}{{\mathbf{Q}}}
\newcommand{\R}{{\mathbf{R}}}
\newcommand{\Z}{{\mathbf{Z}}}

\newcommand{\Qbar}{{\overline{\Q}}}

\newcommand{\cL}{{\mathcal{L}}}

\newcommand{\cO}{{\mathcal{O}}}

\newcommand{\fo}{{\mathfrak{o}}}

\newcommand{\hh}{{\hat{h}}}

\newcommand{\htau}{{\hat{\tau}}}

\newcommand{\tf}{{\tilde{f}}}

\newcommand{\tG}{{\tilde{G}}}

\newcommand{\tV}{{\tilde{V}}}

\renewcommand{\a}{\alpha}

\renewcommand{\d}{\delta}
\newcommand{\la}{\lambda}

\newcommand{\e}{\varepsilon}

\newcommand{\p}{\psi}

\newcommand{\ie}{i.e.\ }

\newcommand{\ord}{\operatorname{ord}}

\begin{document}

\title[Canonical heights]{Canonical heights for plane polynomial maps of small topological degree}
\date{\today}
\author{Mattias Jonsson}
\address{Dept of Mathematics, University of Michigan, Ann Arbor \\ MI 48109-1043\\ USA}
\email{mattiasj@umich.edu}

\author{Elizabeth Wulcan}
\address{Dept of Mathematics, Chalmers University of Technology and the University of Gothenburg\\ SE-412 96 G{\"o}teborg, Sweden}
\email{wulcan@chalmers.se}



\begin{abstract}
  We study canonical heights for 
  plane polynomial mappings of small topological degree.
  In particular, we prove that for points of canonical height zero,
  the arithmetic degree is bounded by the topological degree and hence
  strictly smaller than the first dynamical degree.
  The proof uses the existence, proved by Favre and the first author, 
  of certain compactifications of the plane adapted to the dynamics.
\end{abstract}

\maketitle
%
%
%
%
%
%
\section{Introduction}
J.~Silverman~\cite{Silverman11} recently proposed a number of conjectures on 
the growth of heights and degrees under iterates
of rational selfmaps of projective space, and proved them for monomial maps.
Here we study the growth of heights for a large class of plane polynomial maps.

Consider a polynomial mapping $f:\A^2\to\A^2$ defined
over the field $\Qbar$ of algebraic numbers.
The \emph{first dynamical degree} $\la_1$ 
is defined by
\begin{equation*}
  \la_1:=\lim_{n\to\infty}(\deg f^n)^{1/n}.
\end{equation*}
The \emph{second dynamical degree} $\la_2$ is the number of preimages 
under $f$ of a general closed point in $\A^2$.
It follows from B\'ezout's Theorem that $\la_2\le\la_1^2$.
Following Guedj~\cite{guedjpoly} we say that $f$ 
has \emph{small topological degree} if 
$\la_2<\la_1$. 

Let $h$ be the standard logarithmic height on $\P^2(\Qbar)\supseteq\A^2(\Qbar)$. 
\begin{MainThm}
  Let $f:\A^2\to\A^2$ be a polynomial mapping of
  small topological degree, $\la_2<\la_1$. 
  Then the limit 
  \begin{equation*}
    \hh:=\lim_{n\to\infty}\la_1^{-n}h\circ f^n
  \end{equation*}
  exists, and is finite, pointwise on $\A^2(\Qbar)$. 
  We have $\hh\not\equiv0$ and $\hh\circ f=\la_1\hh$.
  Further, if $P\in\A^2(\Qbar)$ and $\hh(P)=0$, then
  \begin{equation}\label{e201}
    \limsup_{n\to\infty}h(f^n(P))^{1/n}\le\la_2<\la_1.\tag{*}
  \end{equation}
  If $f$ is moreover a polynomial automorphism, then 
  $\hh(P)=0$ iff $P$ is periodic.
\end{MainThm}
As in~\cite{Silverman11} we call $\hh$ the \emph{canonical height} associated to $f$,
whereas the left hand side of~\eqref{e201} is the \emph{arithmetic degree} of the point $P$.
Our Main Theorem says that for points of canonical height zero, the 
height along the orbit grows relatively slowly.

A related Conjecture~3 in~\cite{Silverman11} states that $\hh(P)>0$ whenever 
$P$ has Zariski dense orbit in $\A^2$.
By the Main Theorem, this holds for polynomial automorphisms, but we have not been 
able to establish it for noninvertible polynomial mappings of small topological 
degree. For such maps, the locus $\hh=0$ may contain points 
for which $h\circ f^n$ grows exponentially, see~\S\ref{S111}. 

The basic arithmetic dynamics of 
plane polynomial automorphisms
is well understood, thanks to work of 
Silverman~\cite{Silverman94}, 
Denis~\cite{Denis95},
Marcello~\cite{Marcello00,Marcello03},
Kawa\-guchi~\cite{Kawaguchi06,Kawaguchi09},
Lee~\cite{Lee09},
Ingram~\cite{Ingram11} 
and others. 
In particular, the Main Theorem above and Conjecture~3 in~\cite{Silverman11}
were already known\footnote{The existence of $\hh$ as a limit rather than a
  $\limsup$ seems to be new, however.}
for polynomial automorphisms $f$ with 
$\la_1(f)>1$.
The existing proofs make use of the inverse map $f^{-1}$ and also 
rely crucially on the Friedland-Milnor classification~\cite{FriedlandMilnor}, which 
shows that, up to conjugation, $f$
is a composition of generalized H\'enon maps
and in particular \emph{regular} in the sense of~\cite{Sibony}. 
In fact, the two results above are true for regular polynomial automorphisms 
of any dimension, see~\cite[Theorem~36]{Silverman11}.

For noninvertible maps, no algebraic classification is known. 
Instead we exploit a result by Favre and the first
author~\cite{dyncomp} which shows that maps of small topological degree 
always admit compactifications that are well adapted to the dynamics. 
We refer to~\S\ref{S104} for a precise statement.
Given such a compactification, we work locally with a given absolute value and estimate
the growth of the local height under iteration. When the absolute value is Archimedean,
this was essentially carried out in~\cite[\S7]{dyncomp}, adapting techniques from the early work on the complex H\'enon map, see~\cite{Hubbard86,HO,FriedlandMilnor,BS1,FSHenon}.

The existence of a compactification adapted to the dynamics is proved in~\cite{dyncomp}
using the induced dynamics on a suitable space of valuations, 
see also~\cite{valtree,eigenval,dynberko}.
In order to address Silverman's Conjecture~3 in our setting, one would likely
have to refine this valuative analysis, something which is 
beyond the scope of the present paper.

The ergodic theory of polynomial maps of small topological degree over
the complex numbers is studied in detail in~\cite{DDG1,DDG2,DDG3}. 
It would be interesting to see if these results have non-Archimedean 
or arithmetic analogues.

The paper is organized as follows. In~\S\ref{S101} we recall some facts concerning 
dynamical degrees and heights. We also state the key result from~\cite{dyncomp}
on the existence of a compactification adapted to the dynamics. 
The Main Theorem is proved in~\S\ref{S106}.

%
%
%
%
%
%
\section{Background}\label{S101}
Unless indicated otherwise, we work over the field $\Qbar$
of algebraic numbers.
%
%
%
%
\subsection{Admissible compactifications}\label{S102}
We use the standard embedding $\A^2\hookrightarrow\P^2$
given in coordinates by $(x_1,x_2)\mapsto[1:x_1:x_2]$.
Let $\cL$ be the set of affine functions on~$\A^2$.
\begin{Def}\cite{dyncomp}
  An \emph{admissible compactification} of $\A^2$ is a smooth projective 
  surface $X$ together with a birational morphism 
  $\pi:X\to\P^2$ that is an isomorphism above $\A^2$.
\end{Def}
We can thus view $\A^2$ as a Zariski open subset of $X$.
By the structure theorem for birational morphisms of surfaces, $\pi$
is a finite composition of point blowups, 
so the divisor $X\setminus\A^2$ has normal crossing singularities.

Let $\xi\in X\setminus\A^2$ be a closed point.
The closure in $X$ of at least one of the curves 
$\{x_i=0\}\subseteq\A^2$, $i=1,2$, 
does not contain $\xi$: this is true already when $X=\P^2$.
Pick local coordinates $z_1,z_2\in\cO_{X,\xi}$ such that 
$X\setminus\A^2\subseteq E_1\cup E_2$, locally at $\xi$,
where $E_i=\{z_i=0\}$. 
Write $b_i=-\max_{\ell\in\cL}\ord_{E_i}(\ell)\in\Z_{\ge0}$,
where $\ord_{E_i}(\ell)$ is the order of vanishing of $\ell$ along $E_i$.
Note that $b_i=-\ord_{E_i}(\ell)$ for a general affine function
$\ell\in\cL$. Thus $b_i=0$ iff $E_i\cap\A^2\ne\emptyset$.
We can write
\begin{equation}\label{e101}
  x_i=z_1^{-b_1}z_2^{-b_2}\p_i,
\end{equation}
where $\p_i\in\cO_{X,\xi}$ and where $\p_i(\xi)\ne0$ for at least one $i$.

Similarly, if $f:\A^2\to\A^2$ is any dominant polynomial mapping, then 
\begin{equation}\label{e107}
  f^*x_i=z_1^{\ord_{E_1}(f^*\ell)}z_2^{\ord_{E_2}(f^*\ell)}\chi_i,
\end{equation}
for a general affine function $\ell\in\cL$, where $\chi_i\in\cO_{X,\xi}$.
In this case, we do not claim that $\chi_i(\xi)\ne0$ for some $i$.
%
%
%
%
\subsection{Degree growth}\label{S103}
See~\cite{guedjpoly,eigenval,dyncomp} for more details on what follows.
Let $f:\A^2\to\A^2$ be a dominant polynomial mapping defined over $\Qbar$.
The \emph{degree} $\deg f$ of $f$ is the degree of $f^*\ell$ for a general
affine function $\ell\in\cL$.
Thus $\deg f=\max_i\deg f^*x_i$.

It is easy to see that 
$\deg f^{n+m}\le(\deg f^n)(\deg f^m)$, hence the limit
$\la_1:=\la_1(f):=\lim_{n\to\infty}(\deg f^n)^{1/n}$ exists. It is called the 
\emph{first dynamical degree} of $f$.
The \emph{second dynamical degree} $\la_2:=\la_2(f)$ is 
the topological degree of $f$, \ie 
the number of preimages of a general closed point in $\A^2$.
Note that $\la_i(f^n)=\la_i(f)^n$. 
It follows from B\'ezout's Theorem that $\la_2\le\la_1^2$.
While the degree $\deg f$ depends on the choice of embedding
$\A^2\hookrightarrow\P^2$, the dynamical degrees $\la_i(f)$ do not.

The degree growth sequence $(\deg f^n)_{n=0}^\infty$ of plane
polynomial maps was studied in detail in~\cite{eigenval,dyncomp};
see also~\cite{deggrowth}. 
In particular, we have
\begin{Thm}\cite[Theorem~A'.]{eigenval}\label{T102}
  If $\la_1>1$, then $\deg f^n\sim n^l\la_1^n$ as $n\to\infty$, 
  where $l\in\{0,1\}$.
  Further, $l=0$ unless $\la_2=\la_1^2$ and $f$ is conjugate to a skew
  product.
\end{Thm}
%
%
%
%
\subsection{Dynamical compactifications}\label{S104}
The following result from~\cite{dyncomp} plays a key role in 
the proof of the Main Theorem.
\begin{Thm}\label{T101}
  Let $f:\A^2\to\A^2$ be a polynomial map of 
  small topological degree: $\la_2<\la_1$.
  Then, for every $\e>0$ there exist an integer $n_0\ge1$,
  an admissible compactification $X$ of $\A^2$ and a 
  decomposition $X\setminus\A^2=Z^+\cup Z^-$ into 
  (possibly reducible) curves $Z^+$, $Z^-$ without common
  components, such that the following properties hold.
  \begin{itemize}
  \item[(1)]
    If $F$ is any irreducible component of $Z^-$ and $\ell\in\cL$ is
    a general affine function, then
    \begin{equation}\label{e109}
      \ord_F(f^{n_0*}\ell)\ge(\la_2+\e)^{n_0}\ord_F(\ell).
    \end{equation}
  \item[(2)]
    The extension $f^{n_0}:X\dashrightarrow X$ of $f^{n_0}$ as a rational map 
    is regular at any point on $Z^+$.
  \item[(3)]
    There exists a closed point $\xi_+\in Z^+\setminus Z^-$
    such that $f^{n_0}(Z^+)=\{\xi_+\}$ and $f(\xi_+)=\xi_+$. 
    Further, we are in one of the following two cases:
    \begin{itemize}
    \item[(a)]
      $\la_1\not\in\Q$, there are two irreducible components 
      $E_1$, $E_2$ of $Z^+$ containing $\xi_+$ and, locally at $\xi_+$, we have 
      $f^*E_i=a_{i1}E_1+a_{i2}E_2$, where $a_{ij}\in\N$, 
      and the $2\times2$ matrix $(a_{ij})$ has spectral radius $\la_1$;
    \item[(b)]
      $\la_1\in\N$, there is a unique irreducible component $E$ of $Z^+$
      containing $\xi_+$, $f(E)=\{\xi_+\}$ and, locally at $\xi_+$, we have 
      $f^*E=\la_1E$.
    \end{itemize}
  \end{itemize}
\end{Thm}
The theorem above corresponds 
to~\cite[Lemma~7.3]{dyncomp}.\footnote{The inequality~\eqref{e109} 
  is reversed in~\cite{dyncomp} but this is a typo. The right hand side of~\eqref{e109} is negative.}
The fact that we may assume $\la_1\not\in\Q$ in
case~(a) follows from Proposition~2.5 and the proof of Theorem~3.1 in~\cite{dyncomp}.

Note that everything in~\cite{dyncomp} is stated over the complex numbers, 
and the assertions in~(a) and~(b) are slightly more precise than what is
written here: they give normal forms in (possibly transcendental) local coordinates,
based on the work of Favre~\cite{FavreRigid}.
However, the main analysis in~\cite{dyncomp} is purely algebraic. When working 
over a general algebraically closed field of characteristic zero one obtains
Theorem~\ref{T101}. 
%
%
%
%
\subsection{Absolute values}\label{S113}
Consider an admissible compactification $X$ of $\A^2$.
Let $K$ be a number field such that $X$ and the map 
$\pi:X\to\P^2$ are defined over $K$. 

Let $M_K$ be the set of (normalized) absolute values on $K$.
For each $v\in M_K$, let $K_v$ be the completion.
Write $X(K)$ for the set of $K$-rational points of $X$.
Similarly, $X(K_v)$ is the set of $K_v$-rational points of 
$X\otimes_KK_v$. 
Define $\A^2(K)$ and $\A^2(K_v)$ in the same way
and equip all these spaces with the topology induced by $v$.
Then $X(K_v)$ is compact and contains $X(K)$ as a dense subset. 
Further, $\A^2(K)$ is dense in $X(K)$, hence also in $X(K_v)$.
This follows since every point
in $X$ admits a Zariski open neighborhood $U$ isomorphic to affine 2-space
such that $U\setminus\A^2$ is a proper Zariski closed subset.
%
%
%
%
\subsection{Heights}\label{S105}
See~\cite[\S3.1]{SilvermanBook} or~\cite[{\S}3]{LangDioph} 
for an introduction to heights.
Let $K$ be a number field. 
For any absolute value $v\in M_K$ define a \emph{local height} 
$\tau_v:\A^2(K)\to\R_+$ by
\begin{equation*}
  \tau_v
  =\log^+\max\{|x_1|_v,|x_2|_v\}
  =\log\max\{1,|x_1|_v,|x_2|_v\}
\end{equation*}
and define the (absolute, global) \emph{logarithmic height} 
$h=h_K:\A^2(K)\to\R_+$ by
\begin{equation*}
  h=\sum_{v\in M_K}\frac{[K_v:\Q_v]}{[K:\Q]}\tau_v.
\end{equation*}
If $L/K$ is a finite extension, then 
$h_L=h_K$ on $\A^2(K)$, see~\cite[Proposition~3.4]{SilvermanBook}.

Let us express the local height $\tau_v$ in local coordinates at
infinity. Consider an admissible compactification $X$ of $\A^2$
and a closed point $\xi\in X\setminus\A^2$. 
Let $L\supseteq K$ be a number field over which $X$ and $\xi$ are
defined and extend $v$ as an absolute value on $L$.
Pick local coordinates $z_i\in\cO_{X,\xi}$, $i=1,2$, defined
over $L$ such that $X\setminus\A^2\subseteq E_1\cup E_2$,
locally at $\xi$, where $E_i=\{z_i=0\}$. 
Given $\d_i>0$ define
\begin{equation}\label{e102}
  \Omega_{v,\d}:=\{P\in\A^2(L_v)\mid |z_i(P)|_v<\d_i, i=1,2\}.
\end{equation}
By~\S\ref{S113} we have $\Omega_{v,\d}\cap\A^2(L)\ne\emptyset$.
It follows from~\eqref{e101} that if $0<\d_i\ll1$, then 
\begin{equation}\label{e103}
  \tau_v=-\sum_{i=1}^2b_i\log|z_i|_v+O(1)
\end{equation}
on $\Omega_{v,\d}$.
Similarly, it follows from~\eqref{e107} that if $f:\A^2\to\A^2$
is a dominant polynomial mapping defined over $L$, then 
\begin{equation}\label{e108}
  \tau_v\circ f\le\sum_{i=1}^2\ord_{E_i}(f^*\ell)\log|z_i|_v+O(1)
\end{equation}
on $\Omega_{v,\d}$, where $\ell\in\cL$ is a general affine function.
%
%
%
%
\subsection{Canonical height and arithmetic degree}\label{S114}
Consider a dominant polynomial mapping 
$f:\A^2\to\A^2$ defined over $\Qbar$.
For simplicity assume $\la_1=\la_1(f)>1$.
Silverman defines the \emph{arithmetic degree} on $\A^2(\Qbar)$ by
\begin{equation*}
  \a:=\a_f:=\limsup_{n\to\infty}(h\circ f^n)^{1/n}
\end{equation*}
and shows that $\a\le\la_1$, see~\cite[Proposition~12]{Silverman11}.

By Theorem~\ref{T102} we have $\deg(f^n)\sim n^l\la_1^n$,
where $l\in\{0,1\}$. Further, $l=0$ when $\la_2<\la_1^2$.
Silverman defines the \emph{canonical height} on $\A^2(\Qbar)$ by
\begin{equation*}
  \hh:=\hh_f:=\limsup\frac1{n^l\la_1^n}h\circ f^n
\end{equation*}
and shows in~\cite[Proposition~19~(d)]{Silverman11} that 
$\a(P)<\la_1$ implies $\hh(P)=0$. 
Inspired by Silverman's work, we state
\begin{Conj}\label{C104}
 We have $\hh(P)<\infty$ for all $P\in\A^2(\Qbar)$.
 Moreover, if $\la_2<\la_1^2$, then $\hh(P)=0$ implies
 $\a(P)<\la_1$.
\end{Conj} 
The second part of this conjecture would provide
a converse to Proposition~19~(d) in~\cite{Silverman11},
whereas the first part would give an affirmative answer to 
a special case of Question~18 in~\textit{loc.\ cit}.

Our Main Theorem gives a positive answer to Conjecture~\ref{C104} when
$\la_2<\la_1$ and shows that, under this assumption,
the $\limsup$ in the definition of $\hh$  is in fact a limit.
On the other hand, the following example shows that the 
last statement in Conjecture~\ref{C104} may fail when 
$\la_2=\la_1^2$.
\begin{Example}\cite[Example~16]{Silverman11}.\label{E104}
  If $f(x,y)=(x^2,xy^2)$, then $\la_1=2$, $\la_2=4$ and 
  $\deg f^n=(n+2)\cdot2^{n-1}$, so $l=1$.
  For $P=(2,0)$ we have $h(f^n(P))=2^n\log 2$.
  In particular, $\hh(P)=0$ but $\a(P)=\la_1=2$.
\end{Example}
See also~\cite{KS12} for definitions and results of canonical heights and
arithmetic degrees for rational selfmaps of more general algebraic varieties.
%
%
%
%
%
%
\section{Proof of the Main Theorem}\label{S106}
Fix a polynomial mapping $f:\A^2\to\A^2$
defined over $\Qbar$ and of small topological degree: $\la_2<\la_1$. 
Let $K_0=K_0(f)$ be a number field over which $f$ is defined.
%
%
%
%
\subsection{Growth of local heights}\label{S107}
\begin{Prop}\label{P102}
  We may choose $K_0$ above such that if $K\supseteq K_0$ and 
  $v\in M_K$ is any normalized absolute value, then:
  \begin{itemize}
    \item[(i)]
      the limit $\htau_v:=\lim_{n\to\infty}\la_1^{-n}\tau_v\circ f^n$ 
      exists, and is finite, pointwise on $\A^2(K)$; 
    \item[(ii)]
      the set $U^+_v:=\{\htau_v>0\}\subseteq\A^2(K)$ is open and nonempty;
    \item[(iii)]
      $\limsup_{n\to\infty}(\tau_v\circ f^n)^{1/n}\le\la_2$ on $\A^2(K)\setminus U^+_v$.
    \end{itemize}
\end{Prop}
To prove this, pick $0<\e<\la_1-\la_2$. 
Given this $\e$, let $X$, $\xi_+$ and $n_0$ be as in Theorem~\ref{T101}
and choose a number field 
$L=L(\e)\supseteq K$ such that $f$ and $X$ are defined over $L$
and the closed point $\xi_+$ is $L$-rational.
Extend $v$ as an absolute value on $L$.
\begin{Lemma}\label{L101}
  There exists an open subset 
  $V^+_v\subseteq\A^2(L_v)$ such that 
  $f(V^+_v)\subseteq V^+_v$ and
  \begin{itemize}
  \item[(i)]
    the limit $\htau_v:=\lim_{n\to\infty}\la_1^{-n}\tau_v\circ f^n$ 
    exists pointwise on $V^+_v$ and $0<\htau_v<\infty$;
  \item[(ii)]
    $\A^2(L)\cap V^+_v\ne\emptyset$;
  \item[(iii)]
   $\tau_v\circ f^{n_0}\le(\la_2+\e)^{n_0}\tau_v+O(1)$ on $\A^2(L_v)\setminus V^+_v$,
   with $n_0$ as in Theorem~\ref{T101}.
  \end{itemize}
\end{Lemma}
This result is an analogue of~\cite[Lemma~7.2]{dyncomp}, which in 
turn is modeled on results for the complex H\'enon map. 
See also~\cite[Lemma~2.1]{Ingram11}.

Granted Lemma~\ref{L101} we conclude the proof of Proposition~\ref{P102}
as follows. Set
\begin{equation}\label{e114}
  U^+_v:=\A^2(K)\cap\bigcup_{n\ge0}f^{-n}(V^+_v).
\end{equation}
By Lemma~\ref{L101}~(i), 
the limit defining $\htau_v$ exists on the open set $U^+_v$
and $0<\htau_v<\infty$ there. 
Now suppose $P\in\A^2(K)\setminus U^+_v$. 
Since $f^{jn_0}(P)\not\in V^+_v$ for all $j\ge0$,
Lemma~\ref{L101}~(iii) yields
$\limsup_{j\to\infty}(\tau_v(f^{jn_0}(P)))^{1/jn_0}\le\la_2+\e$. 
Using the bound $\tau_v\circ f\le(\deg f)\tau_v+O(1)$
on $\A^2(K)$ 
(see~\cite[(3.6)~p.92]{SilvermanBook})
we obtain 
\begin{equation*}
  \limsup_{n\to\infty}\tau_v(f^{n}(P))^{1/n}\le\la_2+\e<\la_1.
\end{equation*}
Thus $\htau_v$ is well defined and finite everywhere on $\A^2(K)$
and the set $U^+_v$ is characterized, independently of $L$, as 
$U^+_v=\{\htau_v>0\}$. Letting $\e\to0$ we obtain~(iii).

It only remains to prove that $U^+_v$ is nonempty, assuming that 
$K_0$ is well chosen and $K\supseteq K_0$. 
For this, it suffices to consider the case $K=K_0$. 
Fix $\e_0$ with $0<\e_0<\la_1-\la_2$,
apply Theorem~\ref{T101} with $\e=\e_0$ and 
pick $K_0$ so that all the relevant data in that theorem 
(as well as $f$) are defined over $K_0$.
Given $v\in M_{K_0}$, we now apply Lemma~\ref{L101} with $L=K=K_0$.
We obtain $U^+_v\supseteq\A^2(K_0)\cap V^+_v\ne\emptyset$ as desired.
\begin{proof}[Proof of Lemma~\ref{L101}]
  We first analyze the situation locally at infinity 
  near the fixed point $\xi_+$. 
  Pick local coordinates $(z_1,z_2)$ at $\xi_+$ such that 
  $X\setminus\A^2\subseteq E_1\cup E_2$, locally at $\xi$,
  where $E_i=\{z_i=0\}$.
  Given $0<\d_i\ll1$ define $\Omega_{v,\d}$ as in~\eqref{e102}.
  We claim that for suitable choices of $\d_i$ we have 
  $f(\Omega_{v,\d})\subseteq\Omega_{v,\d}$ and that 
  $\la_1^{-n}\tau_v\circ f^n$ converges pointwise on $\Omega_{v,\d}$
  to a strictly positive function $\htau_v$. 
  To see this, we consider the two cases~(a) and~(b) in
  Theorem~\ref{T101}~(3) separately.
  
  In case~(b) we have $E_1\cap\A^2=\emptyset\ne E_2\cap\A^2$ and
  \begin{equation}\label{e104}
    f^*z_1=z_1^{\la_1}\phi_1
    \quad\text{and}\quad
    f^*z_2=z_1\phi_2,
  \end{equation}
  where $\phi_i\in\cO_{X,\xi_+}$ and $\phi_1(\xi_+)\ne0$.
  For $0<\d_1\ll\d_2\ll1$ it follows, using $\la_1\ge2$, that 
  $f(\Omega_{v,\d})\subseteq\Omega_{v,\d}$.
  We claim that as $n\to\infty$, the functions $s_n:=\la_1^{-n}\log|f^{n*}z_1|_v$ 
  converge pointwise on $\Omega_{v,\d}$ to a function $s_\infty<0$.
  Indeed,~\eqref{e104} yields $s_{n+1}=s_n+\la_1^{-(n+1)}\log|\phi_1\circ f^n|_v$
  and hence 
  \begin{equation*}
    s_n\to s_\infty:=\log|z_1|_v+\sum_0^\infty\la_1^{-(j+1)}\log|\phi_1\circ f^j|_v
    \quad\text{as $n\to\infty$}.
  \end{equation*}
  Here the series is uniformly bounded
  on $\Omega_{v,\d}$, so $s_\infty<0$. 
  Clearly $s_\infty\circ f=\la_1s_\infty$.
  Using~\eqref{e103} we can rewrite this in terms of the local height.
  Note that $b_1>0=b_2$ since $E_1\cap\A^2=\emptyset\ne E_2\cap\A^2$.
  Thus
  $\tau_v=-b_1\log|z_1|_v+O(1)$  on $\Omega_{v,\d}$. 
  It is then clear that 
  $\la_1^{-n}\tau_v\circ f^n$ converges pointwise on 
  $\Omega_{v,\d}$ to the function $\htau_v:=-b_1 s_\infty>0$.
  
  In case~(a) we instead have $E_i\cap\A^2=\emptyset$ for $i=1,2$ and
  \begin{equation}\label{e105}
    f^*z_i=z_1^{a_{i1}}z_2^{a_{i2}}\phi_i
  \end{equation}
  for $i=1,2$, where $a_{ij}\in\N$, 
  $\phi_i\in\cO_{X,\xi_+}$ and $\phi_i(\xi_+)\ne0$. 
  The matrix $(a_{ij})$ has irrational spectral radius $\la_1$ and hence
  admits an eigenvector $\zeta=(\zeta_1,\zeta_2)$ with $\zeta_i>0$.
  Set $\d_i=\kappa^{\zeta_i}$ with $0<\kappa\ll1$.
  Since $\la_1>1$ it follows that 
  $f(\Omega_{v,\d})\subseteq\Omega_{v,\d}$.
  We have
  \begin{equation}\label{e106}
    f^{n*}z_i=z_1^{a_{i1}^{(n)}}z_2^{a_{i2}^{(n)}}
    \prod_{l=0}^{n-1}(\phi_1^{a_{i1}^{(l)}}\phi_2^{a_{i2}^{(l)}})\circ f^{n-1-l},
  \end{equation}
  where, for $0\le l\le n$, $a_{ij}^{(l)}$ are the entries in the matrix $A^l$.
  By the Perron-Frobenius Theorem, $\la_1^{-n}A^n$ converges
  as $n\to\infty$ to a matrix with strictly positive coefficients.
  Using~\eqref{e106} we see that for $i=1,2$,
  the sequence $(s_{i,n})_{n=1}^\infty$ of functions on $\Omega_{v,\d}$,
  defined by $s_{i,n}=\la_1^{-n}\log|f^{n*}z_i|_v$, converges pointwise
  to a strictly negative function $s_{i,\infty}$.
  Clearly $s_{i,\infty}\circ f=\la_1 s_{i,\infty}$.
  Using~\eqref{e103}, this implies that 
  $\la_1^{-n}\tau_v\circ f^n$ converges pointwise on 
  $\Omega_{v,\d}$ to the function $\htau_v:=-\sum_{i=1}^2b_i s_{i,\infty}$.
  Since $b_i>0$ we again have $\htau_v>0$.
  
  Set $V^+_v:=f^{-n_0}(\Omega_{v,\d})$. Clearly $V^+_v$ is open in $\A^2(L_v)$
  and $f(V^+_v)\subseteq V_v^+$.
  Further, $\la_1^{-n}\tau_v\circ f^n$ converges pointwise on 
  $V^+_v$ to a function $\htau_v>0$.
  By~\S\ref{S113}, we have 
  $V^+_v\cap\A^2(L)\supseteq\Omega_{v,\d}\cap\A^2(L)\ne\emptyset$.
  
 \smallskip
  It remains to prove the estimate in~(iii). 
  Let $\tV^+_v\subseteq X(L_v)$ be an open set such that 
  $\A^2(L_v)\cap\tV^+_v=V^+_v$ and $Z^+(L_v)\subseteq\tV^+_v$.
  Such a set can be constructed as follows.
  Following~\eqref{e102} set
  \begin{equation*}
    \tilde{\Omega}_{v,\d}:=\{P\in X(L_v)\mid |z_i(P)|_v<\d_i, i=1,2\},
  \end{equation*}
  so that $\tilde{\Omega}_{v,\d}\cap\A^2(L_v)=\Omega_{v,\d}$.
  Recall that the extension of $f^{n_0}$ as a rational selfmap of $X$
  is regular at every point of $Z^+$ and that $f^{n_0}(Z^+)=\{\xi^+\}$.
  We can therefore set
  \begin{equation*}
    \tV^+_v:=(f^{n_0})^{-1}(\tilde{\Omega}_{v,\d})\setminus I,
  \end{equation*}
  where $I$ is the indeterminacy set of $f^{n_0}:X\dashrightarrow X$.

  In order to prove~(iii) it suffices, by compactness, to prove the estimate
  \begin{equation}\label{e113} 
    \tau_v\circ f^{n_0}\le(\la_2+\e)^{n_0}\tau_v+O(1)
  \end{equation}
  in a neighborhood of any point $\eta\in X(L_v)\setminus\tV^+_v$.
  This is clear if $\eta\in\A^2(L_v)$, so we may assume
  $\eta\not\in\A^2(L_v)$.
  Then $\eta\in Z^-(L_v)\setminus Z^+(L_v)$.
  Pick local coordinates $(w_1,w_2)$ at $\eta$, defined over $L_v$,
  such that $X(L_v)\setminus\A^2(L_v)\subseteq F_1(L_v)\cup F_2(L_v)$ 
  locally at $\eta$, 
  where $F_i=\{w_i=0\}$.
  Let $\tG_\eta$ be a small neighborhood of $\eta\in X(L_v)$.
  On $\tG_\eta\cap\A^2(L_v)$ we then have 
  \begin{align*}
    \tau_v\circ f^{n_0}
    &\le\sum_{i=1}^2\ord_{F_i}(f^{n_0*}\ell)\log|w_i|_v+O(1)\notag\\
    &\le(\la_2+\e)^{n_0}\sum_{i=1}^2\ord_{F_i}(\ell)\log|w_i|_v+O(1)\notag\\
    &\le(\la_2+\e)^{n_0}\tau_v+O(1),
  \end{align*}
  where $\ell\in\cL$ is a general affine function.
  Here the first and third inequality follow from~\eqref{e108}
  and~\eqref{e103}, respectively. The second inequality
  results from~\eqref{e109} when $F_i\subseteq Z^-$ and is 
  trivial otherwise, since $\ord_{F_i}(f^*\ell)\le 0=\ord_{F_i}(\ell)$
  in this case.

  Thus~\eqref{e113} holds, which completes the proof.
\end{proof}
%
%
%
%
\subsection{Growth of heights}\label{S108}
Let $K_0$ be as in Proposition~\ref{P102}.
Pick any point $P\in\A^2(\Qbar)$ and let $K\supseteq K_0$
be a number field such that $P$ is $K$-rational.
We claim that there exists a finite subset 
$M_K'=M_K'(P)\subseteq M_K$ such that
\begin{equation}\label{e111}
  \tau_v(f^n(P))=0
  \quad\text{for all $v\in M_K\setminus M'_K$ and all $n\ge 0$}. 
\end{equation}
To see this, note that $f^*x_i\in K[x_1,x_2]$, $i=1,2$, so for 
all but finitely many $v\in M_K$ we will have 
$f^*x_i\in\fo_v[x_i]$, where $\fo_v=\{a\in K\mid |a|_v\le 1\}$.
For these $v$ we then also have 
$f^{n*}x_i\in\fo_v[x_i]$ for all $n\ge1$.
Further, for all but finitely many $v$ we have 
$x_i(P)\in\fo_v$. Combining these statements and the 
definition of $\tau_v$, we obtain~\eqref{e111}.

It is now easy to conclude. We have 
\begin{equation*}
  h(f^n(P))
  =\sum_{v\in M_K}\frac{[K_v:\Q_v]}{[K:\Q]}\tau_v(f^n(P))
  =\sum_{v\in M'_K}\frac{[K_v:\Q_v]}{[K:\Q]}\tau_v(f^n(P)),
\end{equation*}
so since $M'_K$ is finite, all assertions in the Main Theorem
follow from Proposition~\ref{P102}, except for the claim
about polynomial automorphisms.
%
%
%
%
\subsection{Polynomial automorphisms}\label{S110}
Now assume that $f$ is a polynomial automorphism with $\la_1>1$.
We must prove that if $\hh(P)=0$, then $P$ is periodic.

\medskip
First suppose $f$ is \emph{regular}~\cite{Sibony}
in the sense that the indeterminacy loci of the 
birational maps $f,f^{-1}:\P^2\dashrightarrow\P^2$
are disjoint.

From Propositions~2.2.2 and~2.3.2 in~\cite{Sibony}
we deduce the following facts:
$\deg f=\deg f^{-1}=\la_1$;
the indeterminacy locus of $f$ (resp.\ $f^{-1}$) 
is a single point $\xi_-$ (resp.\ $\xi_+$)
on the line at infinity $L_\infty=\P^2\setminus\A^2$
with $\xi_+\ne\xi_-$; 
$f(L_\infty\setminus\{\xi_-\})=\xi_+$ and
$f^{-1}(L_\infty\setminus\{\xi_+\})=\xi_-$.
  
Let $X$ be the minimal admissible compactification of 
$\A^2$ for which the induced birational map 
$f:X\dashrightarrow\P^2$ is a morphism (\ie without
indeterminacy points). Concretely, $X$ is
obtained from $\P^2$ by successively blowing up the
indeterminacy locus of $f$. Since $f$ is birational, 
$X\setminus\A^2$ has a unique irreducible component
$E_-$ that is mapped onto $L_\infty$ by $f$. 
All other irreducible components are mapped to $\xi_+$.

Set $Z^-:=E_-$ and $Z^+:=X\setminus(\A^2\cup E_-)$.
We claim that a strong version of Theorem~\ref{T101} holds,
with $n_0=1$.
Indeed, note that statement~(2) holds. 
The same is true of statement~(3): we are in case~(b), 
with $E=L_\infty$.
Further, we claim that 
\begin{equation}\label{e112}
  \ord_{E_-}(f^*\ell)=-1
  \quad\text{and}\quad
  \ord_{E_-}(\ell)=-\la_1,
\end{equation}
for a general affine function $\ell$.
Since $\la_1>\la_2=1$, this is stronger than~(1).

To prove~\eqref{e112}, recall that $f(E_-)=L_\infty$. Since
$f$ is an automorphism, this implies 
$\ord_{E_-}(f^*\ell)=\ord_{L_\infty}(\ell)=-1$.
Similarly, 
$\ord_{E_-}(\ell)=\ord_{L_\infty}((f^{-1})^*\ell)=-\deg f^{-1}=-\la_1$.

Fix a number field $K_0$ over which $f$ and $X$ are defined.
Suppose $K\supseteq K_0$ and consider an absolute value $v\in M_K$.
Using~\eqref{e112} we prove Lemma~\ref{L101} with $L=K$ and 
the following estimate, which is stronger than the one in~(iii):
\begin{equation*}
  \tau_v\circ f\le\la_1^{-1}\tau_v+O(1)
\end{equation*}
on $\A^2(K_v)\setminus V^+_v$. Since $\la_1>1$, it follows that 
the sequence $(\tau_v\circ f^n)_{n=0}^\infty$ must be pointwise bounded on 
$\A^2(K)\setminus U^+_v$, where $U^+_v$ is defined as in~\eqref{e114}.

Now fix $P\in\A^2(\Qbar)$ with $\hh(P)=0$.
Pick $K\supseteq K_0$ such that $P\in\A^2(K)$.
From the preceding paragraph and the arguments in~\S\ref{S108} 
it follows that the sequence $(h(f^n(P)))_{n=0}^\infty$ 
is bounded, which by the Northcott Theorem implies that 
$P$ is preperiodic, hence periodic.

\medskip
Finally we treat the general case when $f$ is not necessarily regular.
It follows from~\cite{FriedlandMilnor}
that there exists a polynomial automorphism 
$g:\A^2\to\A^2$ such that $\tf:=g^{-1}fg$ is regular.
Set $D=(\deg g)(\deg g^{-1})$. 
Since $f^n=g\tf^ng^{-1}$ we have $D^{-1}\deg\tf^n\le\deg f^n\le
D\deg\tf^n$ and hence $\la_1(\tf)=\la_1(f)$.

As for the growth of heights, we already know that the limits 
\begin{equation*}
  \hh_f:=\lim_{n\to\infty}\la_1^{-n}h\circ f^n
  \quad\text{and}\quad
  \hh_\tf:=\lim_{n\to\infty}\la_1^{-n}h\circ\tf^n
\end{equation*}
exist, pointwise on $\A^2(\Qbar)$.
Since $\deg g^{\pm 1}\le D$, we have by~\cite[Theorem~3.11]{SilvermanBook} that
\begin{equation*}
  D^{-1}(h\circ\tf^n)+O(1)\le h\circ f^n\circ g\le D(h\circ\tf^n)+O(1)
\end{equation*}
and hence $D^{-1}\hh_\tf\le\hh_f\circ g\le D\hh_\tf$.
In particular, if $\hh_f(P)=0$, then $\hh_\tf(g^{-1}(P))=0$.
By what precedes, $g^{-1}(P)$ is then periodic under $\tf$, so that
$P$ is periodic under $f$.
This concludes the proof of the Main Theorem.
%
%
%
%
\subsection{A non-invertible example}\label{S111}
Let us consider the map
\begin{equation*}
  f(x,y)=(y^2(xy+1),x(xy^3+1)).
\end{equation*}
The topological degree is $\la_2=4$. For example, the point $(0,0)$
has three preimages $(0,0)$,  $(1,-1)$ and $(-1,1)$, with
multiplicities
2, 1 and 1, respectively.
It is easy to find a recursion relation for $(\deg f^n)_{n=0}^\infty$ 
and show that the first dynamical degree is $\la_1=2+\sqrt{7}$.
In particular, $\la_1>\la_2$, so $f$ is of small topological degree.

Note that $f^2(x,0)=(x^2,0)$ and $f^2(0,y)=(0,y^2)$. 
Since $\la_1>\sqrt{2}$, the coordinate axes are contained in the locus 
$\{\hh=0\}$, which therefore contains points 
for which $h\circ f^n$ grows exponentially.
%
%
%
%
\subsection*{Acknowledgment}
We thank 
Dennis Eriksson, 
Shu Kawaguchi,
Joey Lee,
Mircea Musta\c{t}\u{a} 
and
Joe Silverman
for fruitful discussions.
The first author was supported by the NSF and the second author
by the Swedish Research Council.
%
%
%
%
%
%


\end{document}